\documentclass[a4paper,12pt]{article}
\usepackage{amstext,amsmath,amssymb,mathrsfs}
\usepackage{amsthm}
\newcommand{\Eu}{{\rm{E}}_{\bd{u}}}
\newcommand{\Pxi}{{\mathbb{P}}_{\xi}}
\newcommand{\Exi}{{\mathbb{E}}_{\xi}}
\newcommand{\tauu}{\tau_{\bd{u}}}
\newcommand{\Phixi}[1]{\Phi_{\xi}^{#1}}
\newcommand{\Mxi}[1]{\mathcal{M}_{\xi}^{#1}}
\theoremstyle{plain}
\newtheorem{theorem}{Theorem}[section]
\newtheorem{lemma}[theorem]{Lemma}
\newtheorem{prop}[theorem]{Proposition}

\newtheorem{defi}[theorem]{Definition}

\theoremstyle{definition}
\newtheorem{remark}{Remark}

\def\E{{\rm{E}}}
\def\P{{\rm{P}}}
\def\R{\mathbb{R}}

\def\I#1{\mathbb{I}_{#1}}
\def\J{\mathbb{J}}

\newcommand{\bd}[1]{\mbox{\boldmath$#1$}}

\renewcommand\tilde{\widetilde}
\renewcommand\hat{\widehat}
\renewcommand\#{\sharp}
\def\Det{\underset{\substack{(s, t) \in \{ t_1, \ldots, t_M \}^2, \\ (x, y) \in \mathbb{Z}^2}}{{\mathrm{Det}}}}

\newcommand{\Psixi}{\Psi_{\xi}^{\bd{t}}[\bd{f}]}
\newcommand{\hatPhi}{\hat{\Phi}_{a\mathbb{Z}}^k}
\newcommand{\hatM}{\hat{\mathcal{M}}_{a\mathbb{Z}}^k}
\newcommand{\hatMj}{\hat{\mathcal{M}}_{a\mathbb{Z}}^j}

\makeatletter
 
 \@addtoreset{equation}{section}
\makeatother

\title{\bf Noncolliding system \\ of continuous-time random walks}
\author{Syota Esaki\thanks{Department of Mathematics and Informatics, Faculty of Science, Chiba University, 1-33 Yayoi-cho, Inage-ku, Chiba 263-8522, Japan; e-mail: sesaki@graduate.chiba-u.jp}}
\date{\today}

\def\C{\mathbb{C}}
\def\R{\mathbb{R}}
\def\N{\mathbb{N}}

\def\Z{\mathbb{Z}}

\def\rP{{\rm P}}
\def\rE{{\rm E}}

\def\P{{\mathbb P}}
\def\E{{\mathbb E}}

\def\1{{\bf 1}}

\def\mbK{\mathbb{K}}
\def\bK{{\bf K}}

\def\J{\mathbb{J}}

\def\cS{{\cal S}}

\begin{document}

\maketitle


\begin{abstract}
The continuous-time random walk is defined as a Poissonization of discrete-time random walk. We study the noncolliding system of continuous-time simple and symmetric random walks on $\mathbb{Z}$. We show that the system is determinantal for any finite initial configuration without multiple point. The spatio-temporal correlation kernel is expressed by using the modified Bessel functions. We extend the system to the noncolliding process with an infinite number of particles, when the initial configuration has equidistant spacing of particles, and show a relaxation phenomenon to the equilibrium determinantal point process with the sine kernel.  
\end{abstract}

\section{Introduction} 

Eigenvalue distributions of Hermitian random-matrix ensembles
provide typical examples of determinantal point processes (DPPs)
on $\R$ \cite{Meh04,For10}.
The essential characteristic of a DPP is 
repulsive interaction acting between any pair of points
\cite{Sos00,ST03}.
Recently, it is clarified that such negatively correlated 
point processes are useful not only to simulate
energy-level statistics of complex quantum many-particle systems,
but also to describe statistics of sets of items that are diverse; for example,
queries of users and topics in daily news.
A variety of such real-world applications of DPPs in
the machine learning technologies is surveyed in \cite{KT12}.
In particular, the eigenvalue distribution of 
non-Hermitian random matrices called the Ginibre ensemble \cite{Gin65}
has attracted much attention both in pure mathematics and 
in applications, since it gives a DPP on a complex plane $\C$.
The Ginibre-Voronoi tessellation on the plane
has been studied \cite{Gol10}
and its advantage than the classical Poisson-Voronoi tessellation
in the applications to cellular network modeling
is reported \cite{MS12}. See \cite{DFL13} and papers cited therein for simulation algorithms of DPPs.

DPPs originally considered on $\R$ have been also extended
to the spatio-temporal plane $\R \times [0, \infty)$.
Such dynamical extensions of DPPs are called
determinantal processes \cite{KT10,Osa12,KT13,Osa13a,Osa13b}. 
Typical examples of determinantal processes are 
noncolliding diffusions including Dyson's Brownian motion
model with $\beta=2$ \cite{Dys62},
where noncolliding conditions make the systems
be negatively correlated. 
The purpose of the present paper is to introduce a discrete model
defined on a lattice $\Z$, which realizes a determinantal process.
Discretization of models will be useful in applications
in the future.


The continuous-time random walk is defined as a Poissonization of discrete-time random walk. We construct a noncolliding system of continuous-time simple and symmetric random walks on $\Z$ as an $h$-transform discussed by K\"{o}nig, O'Connell, and Roch \cite{KOR}. We show that the system has a determinantal martingale representation \cite{K.dif}. We prove that for any finite initial configuration without multiple point, $\xi(\cdot)=\sum_{j=1}^N \delta_{u_j}(\cdot)$, $u_1 < u_2 < \cdots < u_N$, $u_j \in \Z$, $1 \leq j \leq N \in \N$, the system is determinantal in the sense that all spatio-temporal correlation functions are given by determinants specified by the correlation kernel. The correlation kernel is explicitly determined as
\begin{align}
&\mathbb{K}_{\xi}(s, x; t, y) = \sum_{j=1}^N I_{|x-u_j|}(s)I_{|y-u_j|}(-t) \notag \\
&\quad + \sum_{j=1}^N \sum_{w \in \Z \setminus \{ u_k \}_{k=1}^N} I_{|x-u_j|}(s)I_{|y-w|}(-t)\prod_{\substack{1 \leq \ell \leq N, \\ \ell \neq j}} \frac{w-u_{\ell}}{u_j-u_{\ell}} -\1(s>t)I_{|x-y|}(s-t), \label{1.14}
\end{align}
$(s, x), (t, y) \in [0, \infty) \times \Z$. Here $I_{\nu}(z)$ is the modified Bessel function of the first kind of order $\nu$ defined by \cite{Wat44}
\begin{equation}
I_{\nu}(z) := \left( \frac{z}{2} \right)^{\nu} \sum_{\ell=0}^{\infty} \frac{\left( z/2 \right)^{2\ell}}{\ell!\Gamma(\nu+\ell+1)}, \quad \nu > -1, \label{Idef}
\end{equation}
and $\1(\cdot)$ is an indicator; $\1(\omega)=1$ if $\omega$ is satisfied, and $\1(\omega)=0$ otherwise. 

We extend the system to the noncolliding process with an infinite number of particles, when the initial configuration is given by 
\begin{equation}
\label{xiaZ} \xi_{a\Z}(\cdot)=\sum_{k \in \Z}\delta_{ak}(\cdot),
\end{equation}
having equidistant spacing $a \in \{ 2, 3, \ldots \}$ between particles on $\Z$. We prove that this infinite particle process is also determinantal and the correlation kernel is given by 
\begin{align} 
&\mathbb{K}_{\xi_{a\mathbb{Z}}}(s, x; t, y) \notag \\
&= \sum_{j \in \mathbb{Z}} I_{|x-aj|}(s)\frac{1}{2\pi}\int_{-\pi}^{\pi}d\lambda e^{i\lambda(y/a-j)+t\cos(\lambda/a)} -\1(s>t)I_{|x-y|}(s-t), \label{4.9}
\end{align}
$(s, x), (t, y) \in [0, \infty) \times \Z$. Moreover, we show a relaxation phenomenon to the equilibrium determinantal point process, which is governed by the sine kernel defined on $\Z$ \cite{K.non}.

The paper is organized as follows. In Section 2 we introduce continuous-time random walk and  associated martingales. We construct the noncolliding random walk using the $h$-transform in the sense of Doob in Section 3. In Section 4 we introduce a transformation $\cS$ and give the determinantal martingale representation for the noncolliding random walk. In Section 5 we give the correlation kernel of the noncolliding random walk explicitly and extend  the system to infinite particle processes. 
 
\section{Continuous-time random walk} 

\subsection{Construction} 

Let $\mathbb{Z}$ be the set of all integers and $\eta \in \mathbb{Z}$ be a random variable  with a probability measure $\sigma = (\delta_{-1}+\delta_1)/2$, that is, 
\begin{equation*}
{\rm{Prob}}[ \eta = n ] = \left\{
\begin{array}{rl}
\displaystyle{\frac{1}{2}}, \qquad & n=\pm 1, \cr
& \cr
0, \qquad & n \in \Z \setminus \{-1,1\}.
\end{array} \right.
\label{eqn:eta1}
\end{equation*}
The characteristic function of $\sigma$ is then given by 
\begin{equation}
\label{sigmacha} \hat{\sigma}(z)=\int_{\mathbb{R}}e^{iz\eta}\sigma(d\eta) = \cos z, \quad z \in \mathbb{C}, \  i=\sqrt{-1}. 
\end{equation}
We consider a continuous-time simple and symmetric random walk on $\mathbb{Z}$, which is denoted by $V(t), t \in [0, \infty)$. It is defined as a compound Poisson process such that its characteristic function $\psi_{V(t)}(z)$ is given by \cite{Sat99}, 
\begin{align}
\psi_{V(t)}(z) &:= \rE[e^{izV(t)}] \notag \\
&= \sum_{j=0}^{\infty} e^{-t}\frac{t^j}{j!}\left( \hat{\sigma}(z) \right)^j \notag \\
&= \exp(t(\hat{\sigma}(z)-1)), \quad z \in \mathbb{C}. \label{phiV}
\end{align}
In other words, the present continuous-time random walk is a Poissonization of discrete-time simple and symmetric random walk. In this paper, this process on $\mathbb{Z}$ is simply denoted by RW. 

By definition the generator of RW is given by
\begin{equation*}
L_1 f = \frac{f(x+1)+f(x-1)-2f(x)}{2}
\end{equation*}
for suitable functions $f$. 

\begin{lemma}
For $V( \cdot )$, the transition probability is given by 
\begin{equation}
\label{1.1} p(t, y|x) = \frac{1}{2\pi}\int_{-\pi}^{\pi}dk e^{ik(y-x)} e^{-(1-\cos k)t}, \quad t \in [0, \infty), \ x, y \in \mathbb{Z}. 
\end{equation}
\end{lemma}
\begin{proof}
We consider the co-generator of $L_1$, which is denoted by $L_1^*$. We can see easily $L_1=L_1^*$. Therefore the transition probability of RW is a unique solution of the difference equation
\begin{equation*}
\frac{d}{dt}p(t, y|x) = \frac{1}{2}[p(t, y-1|x)+p(t, y+1|x)-2p(t, y|x)], \quad t \in [0, \infty), \ x, y \in \mathbb{Z}
\end{equation*}
with the initial condition $p(0, y|x)=\delta_{x, y}$. Since the eigenfunction of $L_1^*$ is $\phi_k(x)=e^{ikx}$ with the eigenvalue $\lambda_k=\cos k-1$, $k \in \mathbb{R}$, the integral (\ref{1.1}) solves the differential equation. It is obvious that the initial condition is satisfied by (\ref{1.1}). Then the proof is completed. 
\end{proof}

Using the modified Bessel function (\ref{Idef}) we can give another representation to $p$. 

\begin{lemma} \label{prob.dens}
For $t \in [0, \infty)$, $x, y \in \mathbb{Z}$, 
\begin{equation}
\label{Bessel} p(t, y|x) = e^{-t}I_{|y-x|}(t).
\end{equation}
\end{lemma}
\begin{proof}
By symmetry of the RHS of (\ref{1.1}), $p(t, y|x)=p(t, x|y)$, and hence we can assume $y \geq x$ without loss of generality. We see 
\begin{align}
p(t, y|x) &= e^{-t} \frac{1}{2\pi} \int_{-\pi}^{\pi} dk e^{-ik(y-x)}\sum_{n=0}^{\infty}\frac{t^n}{n!}(\cos k)^n \notag \\
&= e^{-t} \sum_{n=0}^{\infty} \frac{t^n}{n!} \frac{1}{2\pi} \int_{-\pi}^{\pi} dk \left( e^{-ik} \right)^{y-x} \left( \frac{e^{ik}+e^{-ik}}{2} \right)^n. \label{realint}
\end{align}
We rewrite this integral by $k$ into a contour integral by $z=e^{ik}$ along a unit circle. Then (\ref{realint}) is equal to 
\begin{align*}
&e^{-t} \sum_{n=0}^{\infty} \frac{t^n}{n!} \frac{1}{2\pi} \oint \frac{dz}{iz} z^{y-x} \frac{1}{2^n} \left( z+\frac{1}{z} \right)^n \\
&= e^{-t} \sum_{n=0}^{\infty} \frac{t^n}{n!} \frac{1}{2^n} \sum_{\ell=0}^n \binom{n}{\ell} \frac{1}{2\pi i} \oint dz z^{y-x-1+2\ell-n} \notag \\
&= e^{-t} \sum_{\ell=0}^{\infty} \sum_{n=\ell}^{\infty} \left( \frac{t}{2} \right)^n \frac{1}{\ell!(n-\ell)!}\delta_{n, y-x+2\ell} \\
&= 
e^{-t} \left( \frac{t}{2} \right)^{y-x} \sum_{\ell=0}^{\infty} \frac{(t/2)^{2\ell}}{\ell!\Gamma(y-x+\ell+1)} 
= \text{RHS of (\ref{Bessel})}.  
\end{align*}
Thus the proof is completed. 
\end{proof}

\subsection{Associated martingales} 

We introduce a filtration $\{ \mathcal{F}_t : t \in [0, \infty) \}$ for RW defined by $\mathcal{F}_t = \sigma( V(s) : 0 \leq s \leq t)$. 

We perform the Esscher transform with parameter $\alpha \in \mathbb{R}$, $V(\cdot) \to \tilde{V}_{\alpha}(\cdot)$ as 
\begin{equation*}
\tilde{V}_{\alpha}(t) = \frac{e^{\alpha V(t)}}{{\rm{E}}[e^{\alpha V(t)}]}, \quad t \in [0, \infty). 
\end{equation*}
By ${\rm{E}}[e^{\alpha V(t)}] = \psi_{V(t)}(i\alpha) = \exp \{ t(\cosh \alpha-1) \}$,
we have 
\begin{equation*}
\label{2.1} \tilde{V}_{\alpha}(t) = G_{\alpha}(t, V(t))
\end{equation*}
with 
\begin{equation}
\label{2.2} G_{\alpha}(t, x) = \exp \left\{ \alpha x - t( \cosh \alpha -1 ) \right\}, \quad t \in [0, \infty), \quad x \in \mathbb{Z}. 
\end{equation}
\begin{lemma}
$G_{\alpha}(t, V(t))$ is an $\mathcal{F}_t$-martingale for any $\alpha \in \mathbb{R}$. 
\end{lemma}
\begin{proof}
For $s < t$,  
\begin{align*}
{\rm{E}} [ \left. G_{\alpha}(t, V(t)) \right| \mathcal{F}_s ] &= \frac{{\rm{E}}[\left. e^{\alpha V(t)} \right| \mathcal{F}_s ]}{{\rm{E}}[ e^{\alpha V(t)} ]} \\
&= \frac{e^{\alpha V(s)}{\rm{E}}[ e^{\alpha (V(t)-V(s))} ]}{{\rm{E}}[ e^{\alpha V(s)} ]{\rm{E}}[ e^{\alpha (V(t)-V(s))} ]} \\
&= \frac{e^{\alpha V(s)}}{{\rm{E}}[ e^{\alpha V(s)} ]} = G_{\alpha}(s, V(s)).  
\end{align*}
Therefore, $G_{\alpha}(t, V(t))$ is an $\mathcal{F}_t$-martingale. 
\end{proof}
Expansion of (\ref{2.2}) with respect to $\alpha$ around $\alpha = 0$, 
\begin{equation}
\label{2.3} G_{\alpha}(t, x) = \sum_{n=0}^{\infty} m_n(t, x) \frac{\alpha^n}{n!}, 
\end{equation}
determines a series of polynomials of degree $n$, 
\begin{equation}
\label{2.4} m_n(t, x) = \sum_{j=0}^{n} c_n^{(j)}(t)x^j, \quad n \in \mathbb{N}_0.  
\end{equation}
For $n=0, 1, 2, 3, 4, $ they are given by   
\begin{align*}
m_0(t, x) &= 1, \\
m_1(t, x) &= x, \\
m_2(t, x) &= x^2-t, \\
m_3(t, x) &= x^3-3tx, \\
m_4(t, x) &= x^4-6tx^2+3t^2-t.
\end{align*}
They satisfy relations
\begin{equation*}
-\frac{d}{dt} m_n(t, x) = \frac{1}{2}[m_n(t, x+1) -2m_n(t, x) + m_n(t, x-1)], \quad n \in \mathbb{N}_0. 
\end{equation*}
The polynomials $\{ m_n(t, x) \}_{n \in \mathbb{N}_0}$ defined by (\ref{2.3}) are fundamental martingale polynomials in the following sense \cite{K.dif}. 
\begin{lemma}
The polynomials $\{ m_n(t, x) \}_{n \in \mathbb{N}_0}$ are given in the form (\ref{2.4}), in which 
\begin{equation}
\label{coef} c_n^{(j)}(0) = 0, \quad 0 \leq j \leq n-1, \quad c_n^{(n)}(t) \equiv 1, \quad \text{for all $n \in \mathbb{N}_0$.}
\end{equation}
That is $m_n(t, x)$'s are monic polynomials with $m_n(0, x)=x^n$. Moreover, $\{ m_n(t, V(t)) \}_{n \in \N_0}$ are $\mathcal{F}_t$-martingales, $t \in [0, \infty)$. 
\end{lemma}
\begin{proof}
By straightforward calculation we can check (\ref{coef}). We can prove that $\{ m_n(t, V(t)) \}_{n \in \N_0}$ are $\mathcal{F}_t$-martingales from the fact that $G_{\alpha}(t, x)$ is an $\mathcal{F}_t$-martingale for all $\alpha \in \mathbb{R}$. 
\end{proof}

\section{Harmonic transform and noncolliding system} 

Suppose $N \in \mathbb{N}$. We consider an $N$-dimensional RW on $\mathbb{Z}^N$, $\bd{V}(t) = (V_1(t), \ldots, V_N(t))$, $t \in [0, \infty)$, where $V_j(\cdot)$, $1 \leq j \leq N$ are independent copies of $V(\cdot)$. 
We take the initial point $\bd{u} = (u_1, \ldots, u_N) = \bd{V}(0) \in \mathbb{Z}^N$. 
The probability space is denoted by $(\Omega, \mathcal{F}, {\rm{P}}_{\bd{u}})$. 
The expectation is written as ${\rm{E}}_{\bd{u}}$.  
Let 
\begin{equation*}
\mathbb{W}_N = \{ \bd{x} = (x_1, \ldots, x_N) \in \mathbb{R}^N : x_1 < \cdots < x_N \}, 
\end{equation*}
which is the Weyl chamber of type ${\rm{A}}_{N-1}$. Define $\tau_{\bd{u}}$ be the exit time from the Weyl chamber of the RW started at $\bd{u} \in \mathbb{Z}^N \cap \mathbb{W}_N$, 
\begin{equation*}
\label{1.2} \tau_{\bd{u}} = \inf \{ t \geq 0 : \bd{V}(t) \notin \mathbb{W}_N \}.
\end{equation*}
In the present paper, we study the RW conditioned to stay in $\mathbb{W}_N$ forever. That is, $\tau_{\bd{u}} = \infty$ is conditioned. We call such a conditional RW the continuous-time (simple and symmetric) noncolliding RW. 

Let $\mathfrak{M}$ be the space of nonnegative integer-valued Radon measures on $\mathbb{Z}$. We consider the noncolliding RW as a process in $\mathfrak{M}$ and represent it by 
\begin{equation}
\label{1.3} \Xi (t, \cdot) = \sum_{j=1}^N \delta_{X_j(t)}(\cdot), \quad t \in [0, \infty),
\end{equation}
where 
\begin{equation}
\label{1.4} \bd{X}(t) = (X_1(t), \ldots, X_N(t)) \in \mathbb{Z}^N \cap \mathbb{W}_N, \quad t \in [0, \infty).
\end{equation}
The configuration $\Xi(t, \cdot) \in \mathfrak{M}, t \in [0, \infty)$ is unlabeled, while $\bd{X}(t) \in \mathbb{Z}^N \cap \mathbb{W}_N, t \in [0, \infty)$ is labeled. We write the probability measure for $\Xi(t, \cdot), t \in [0, \infty)$ started at $\xi \in \mathfrak{M}$ as $\Pxi$ with expectation $\Exi$, and introduce a filtration $\{ \mathcal{F}(t) : t \in [0, \infty) \}$ defined by $\mathcal{F}(t) = \sigma(\Xi(s) : 0 \leq s \leq t)$. We set $\mathfrak{M}_0 = \{ \xi \in \mathfrak{M}; \xi(\{ x \}) \leq 1$ for any $x \in \mathbb{Z} \}$. 

We write the Vandermonde determinant as 
\begin{equation}
\label{Vdet} h(\bd{x}) = \det_{1 \leq j, k \leq N} [x_j^{k-1}] = \prod_{1 \leq j < k \leq N} (x_k-x_j).  
\end{equation}
We would like to introduce the $h$-transform in the sense of Doob for RW in $\mathbb{W}_N$. For this purpose Corollary 2.2 and Theorem 2.4 in \cite{KOR} proved by K\"{o}nig, O'Connell, and Roch are useful. See also \cite{Koe05,EK08,K.non}. 
Here we rewrite their theorems with modifications to fit the present situation and put the following proposition. Let $\overline{\mathbb{W}_N}$ be $\{ \bd{x} = (x_1, \ldots, x_N) \in \mathbb{R}^N : x_1 \leq \cdots \leq x_N \}$ and $\partial \mathbb{W}_N := \overline{\mathbb{W}_N} \setminus \mathbb{W}_N$. 
\begin{prop}{\rm{(Corollary 2.2 and Theorem 2.4 in \cite{KOR})}} \\
The function $h$ given by (\ref{Vdet}) is harmonic for $\bd{V}(t)$. The restriction of $h$ to $\mathbb{W}_N$ is a strictly positive function. And $h$ vanishes at $\partial\mathbb{W}_N$. 
\end{prop}

By this proposition we can construct the noncolliding RW, $\Xi$, as an $h$-transform of an absorbing RW, $\bd{V}$, in $\mathbb{W}_N$.  

\begin{lemma} \label{WNDoob}
Suppose that $N \in \mathbb{N}$ and $\xi=\sum_{j=1}^N \delta_{u_j}$ with $\bd{u} = (u_1, \ldots, u_N) \in \mathbb{Z}^N \cap \mathbb{W}_N$. Let $t \in [0, \infty)$, $t \leq T < \infty$. For any $\mathcal{F}(t)$-measurable bounded function $F$ we have
\begin{equation}
\label{h-trans} \Exi \left[ F(\Xi(\cdot)) \right] = \Eu \left[ F \left( \sum_{j=1}^N \delta_{V_j(\cdot)} \right) \bd{1}(\tauu > T) \frac{h(\bd{V}(T))}{h(\bd{u})} \right]. 
\end{equation}
\end{lemma}
 
\section{Transformation $\cS$ and Determinantal Martingale Representations}

\subsection{Definition of $\cS$} 

We introduce a transformation,
\begin{equation}
\label{cSdef} \cS\left[ \left. f(W) \right| (t, x) \right] := e^t\sum_{w \in \mathbb{Z}}I_{|w-x|}(-t)f(w), \quad t \in [0, \infty), x \in \mathbb{Z},  
\end{equation}
for $f : \mathbb{Z} \to \mathbb{C}$. 
By the definition, $\cS$ is a linear operator. Note that $W$ in the LHS is a dummy variable, but it will be useful to specify a function $f$ as shown below. 
\begin{lemma} \label{lem.cSrep}
The transformation $\cS$ is related with the characteristic function of RW, (\ref{phiV}) with (\ref{sigmacha}) by
\begin{equation}
\label{cSrep} \cS\left[ \left. e^{\alpha(W-x)} \right| (t, x) \right] = \frac{1}{\psi_{V(t)}(i\alpha)}, \quad \alpha \in \R. 
\end{equation}
\end{lemma}
\begin{proof} 
By the definition (\ref{cSdef}),  
\begin{align}
\text{LHS of (\ref{cSrep})} &= \sum_{w \in \mathbb{Z}}e^tI_{|w-x|}(-t)e^{\alpha(w-x)} 
= \sum_{k \in \mathbb{Z}}e^tI_{|k|}(-t)e^{k\alpha} \notag \\
&= e^t \left\{ \sum_{s=0}^{\infty} e^{-s\alpha}I_s(-t) + \sum_{k=1}^{\infty} e^{k\alpha}I_k(-t) \right\}. \label{cSrep.1}
\end{align}
By the definition (\ref{Idef}) of modified Bessel function, (\ref{cSrep.1}) is equal to 
\begin{align*}
&e^t \left\{ \sum_{s=0}^{\infty} e^{-s\alpha}\sum_{\ell=0}^{\infty} \dfrac{1}{\ell!(\ell+s)!}\left( -\frac{t}{2} \right)^{2\ell+s} + \sum_{k=1}^{\infty} e^{k\alpha} \sum_{\ell=0}^{\infty} \frac{1}{\ell!(k+\ell)!}\left( -\frac{t}{2} \right)^{2\ell+k} \right\} \\
&= e^t \left\{ \sum_{k=-\infty}^0 \sum_{\ell=0}^{\infty} \frac{1}{\ell!(\ell-k)!} \left( -\frac{t}{2} \right)^{2\ell-k}e^{k\alpha} + \sum_{k=1}^{\infty} \sum_{\ell=k}^{\infty} \frac{1}{\ell!(\ell-k)!} \left( -\frac{t}{2} \right)^{2\ell-k} e^{k\alpha} \right\} \\
&= e^t \sum_{\ell=0}^{\infty} \sum_{m=0}^{\infty} \frac{1}{\ell!m!}\left( -\frac{t}{2} \right)^{\ell+m} e^{(\ell-m)\alpha} = e^t \sum_{\ell=0}^{\infty} \frac{1}{\ell!} \left( -\frac{t}{2} \right)^{\ell}e^{\ell\alpha} \sum_{m=0}^{\infty} \frac{1}{m!}\left( -\frac{t}{2} \right)^m e^{-m\alpha} \\
&= e^t \exp\left( -\frac{t}{2}e^{\alpha} \right)\exp\left( -\frac{t}{2}e^{-\alpha} \right) = \exp[-t(\cosh\alpha-1)] = \text{RHS of (\ref{cSrep})}.
\end{align*}
Then the proof is completed. 
\end{proof}

\subsection{Representations of martingales using $\cS$} 

\begin{lemma} \label{lem.2.1}
With the transformation (\ref{cSdef}), the fundamental martingale polynomials for RW, 
$\{ m_n(t, x) \}_{n \in \N_0}$, $t \in [0, \infty)$, have the following representations, 
\begin{equation*} \label{2.19b}
m_n(t, x) = \cS\left[ \left. W^n \right| (t, x) \right], \quad n \in \mathbb{N}_0, \quad t \in [0, \infty), \  x \in \mathbb{R}. 
\end{equation*}
\end{lemma}
\begin{proof}
By Lemma \ref{lem.cSrep}, we can see  
\begin{equation*}
G_{\alpha}(t, x) = \frac{e^{\alpha x}}{\psi_{V(t)}(i\alpha)} = \cS\left[ \left. e^{\alpha W} \right| (t, x) \right], \quad \alpha \in \R, 
\end{equation*}
for $G_\alpha(t, x)$ given by (\ref{2.3}) in Section 2.2. 
We expand the equality with respect to $\alpha$ around $\alpha=0$, and we get the lemma. 
\end{proof}

A direct consequence of Lemma \ref{lem.2.1} is the following. 

\begin{lemma} \label{lem.2.2}
Assume that $f$ is polynomial. Then $\cS\left[ \left. f(W) \right| (t, V(t)) \right]$ is an $\mathcal{F}_t$-martingale. 
\end{lemma}

The transformation (\ref{cSdef}) is extended to the linear transformation of functions of $\bd{x} \in \mathbb{Z}^N$ so that, if $F^{(k)}(\bd{x})=\prod_{j=1}^N f_j^{(k)}(x_j)$, $k=1, 2$, then 
\begin{equation*}
\cS[ \left. F^{(k)}(\bd{W}) \right| \{ (t_{\ell}, x_{\ell}) \}_{\ell=1}^N ] = \prod_{j=1}^N \cS \left[  \left. f_j^{(k)}(W_j) \right| (t_j, x_j) \right], \quad k=1, 2, 
\end{equation*}
and 
\begin{align*}
&\cS \left[ \left. c_1F^{(1)}(\bd{W})+c_2F^{(2)}(\bd{W}) \right| \{(t_{\ell}, x_{\ell})\}_{\ell=1}^N \right] \\
&\quad = c_1\cS\left[ \left. F^{(1)}(\bd{W}) \right| \{(t_{\ell}, x_{\ell})\}_{\ell=1}^N \right] + c_2\cS\left[ \left. F^{(2)}(\bd{W}) \right| \{(t_{\ell}, x_{\ell})\}_{\ell=1}^N \right], 
\end{align*}
$c_1, c_2 \in \mathbb{C}$, for $0 < t_j < \infty$, $1 \leq j \leq N$, where $\bd{W}=(W_1, \ldots, W_N) \in \mathbb{Z}^N$. In particular, if $t_{\ell} = t$, $1 \leq \forall \ell \leq N$, we write $\cS[ \cdot | \{(t_{\ell}, x_{\ell})\}_{\ell=1}^N ]$ simply as $\cS[\cdot | (t, \bd{x})]$ with $\bd{x}=(x_1, \ldots, x_N)$. 
By multiliniearity of determinant, the Vandermonde determinant does not change in replacing $x_i^{k-1}$ by any monic polynomial of $x_j$ of degree $k-1$, $1 \leq j, k \leq N$. Since $m_{k-1}(t, x_j)$ is a monic polynomial of $x_j$ of degree $k-1$,  
\begin{align*}
\frac{h(\bd{V}(t))}{h(\bd{u})} &= \frac{1}{h(\bd{u})} \det_{1 \leq j, k \leq N}[m_{k-1}(t, V_j(t))] \\
&= \frac{1}{h(\bd{u})} \det_{1 \leq j, k \leq N}[ \cS[ \left. W_j^{k-1} \right| (t, V_j(t))] ] \\
&= \cS \left[ \left. \frac{1}{h(\bd{u})} \det_{1 \leq j, k \leq N}[W_j^{k-1}] \right| (t, \bd{V}(t)) \right], 
\end{align*}
where we have used the multilinearity of determinant. Therefore, we have obtained the equality, 
\begin{equation} \label{2.18}
\frac{h(\bd{V}(t))}{h(\bd{u})} = \cS \left[ \left. \frac{h(\bd{W})}{h(\bd{u})} \right| (t, \bd{V}(t)) \right] , \quad t \in [0, \infty). 
\end{equation}

We set $\xi = \sum_{j=1}^N \delta_{u_j} \in \mathfrak{M}_0$ and consider a set of functions of $z \in \mathbb{C}$, 
\begin{equation}
\label{1.10} \Phi_{\xi}^{u_k}(z) = \prod_{\substack{1 \leq j \leq N, \\ j \neq k}} \frac{z-u_j}{u_k-u_j}, \quad 1 \leq k \leq N. 
\end{equation}
For each $1 \leq k \leq N$, the function $\Phi_{\xi}^{u_k}(z)$ is a polynomial of $z$ with degree $N-1$ with zeros at $u_j$, $1 \leq j \leq N$, $j \neq k$ and $\Phixi{u_k}(u_k) = 1$. By lemma \ref{lem.2.2} we can prove that, for each $1 \leq k \leq N$, 
\begin{equation}
\label{1.11} \Mxi{u_k}(t, V_j(t)) := \cS\left[ \left. \Phixi{u_k}(W_j) \right| (t, V_j(t))\right], \quad t \in [0, \infty), \quad 1 \leq j \leq N
\end{equation}
provide independent $\mathcal{F}_t$-martingales. Then we see that for $0 \leq t < \infty$, 
\begin{align}
\rE_{\bd{u}}[\Mxi{u_k}(t, V_j(t))] &= \rE_{\bd{u}}[\Mxi{u_k}(0, V_j(0))] \notag \\
&= \Mxi{u_k}(0, u_j) \notag \\
&= \Phixi{u_k}(u_j) = \delta_{j,k}, \quad 1 \leq j, k \leq N. \label{Mlim0}
\end{align}
Now we consider the determinant identity \cite{KT13}, 
\begin{equation} \label{2.19}
\frac{h(\bd{z})}{h(\bd{u})} = \det_{1 \leq j ,k \leq N} \left[ \Phi_{\xi}^{u_k}(z_j) \right], 
\end{equation}
where $\xi = \sum_{j=1}^N \delta_{u_j}$, $\bd{u} = (u_1, \ldots, u_N) \in \mathbb{W}_N$, $\bd{z} = (z_1, \ldots, z_N) \in \mathbb{C}^N$ and $\Phi_{\xi}^{u_k}(z)$ is given by (\ref{1.10}).  
Using this identity for $h(\bd{W})/h(\bd{u})$ in (\ref{2.18}), we have 
\begin{align}
\frac{h(\bd{V}(t))}{h(\bd{u})} &= \cS \left[ \left. \det_{1 \leq j, k \leq N}[\Phixi{u_k}(W_j)] \right| (t, \bd{V}(t)) \right] \notag \\
&= \det_{1 \leq j, k \leq N} \left[ \cS[\left. \Phixi{u_k}(W_j) \right| (t, V_j(t))] \right] \notag \\
&= \det_{1 \leq j, k \leq N} \left[ \Mxi{u_k}(t, V_j(t)) \right], \quad t \in [0, \infty).  \label{VancS}
\end{align}

\subsection{Determinantal martingales representation} 

Since we consider the noncolliding RW as a process represented by an unlabeled configuration (\ref{1.3}), measurable functions of $\Xi(\cdot)$ are only symmetric functions of $N$ variables, $X_j(\cdot)$, $1 \leq j \leq N$. Then, we obtain the following representation. Following \cite{K.dif}, we call it the determinantal martingale representation (DMR) for the present noncolliding RW. 

\begin{prop} \label{prop.3.1}
Suppose that $N \in \mathbb{N}$ and $\xi=\sum_{j=1}^N \delta_{u_j}$ with $\bd{u} = (u_1, \ldots, u_N) \in \mathbb{Z}^N \cap \mathbb{W}_N$. Let $t \in [0, \infty)$, $t \leq T \in [0, \infty)$. For any $\mathcal{F}(t)$-measurable bounded function $F$ we have
\begin{equation}
\label{DMRXi} \mathbb{E}_{\xi}[F(\Xi(\cdot))] = \rE_{\bd{u}}\left[ F\left( \sum_{j=1}^N\delta_{V_j(\cdot)} \right) \det_{1 \leq j, k \leq N} \left[ \Mxi{u_k}(T, V_j(T)) \right] \right]. 
\end{equation}
That is the present process $(\Xi, \P_{\xi})$ has DMR associated with $(V, \mathcal{M}_{\xi})$, where $\mathcal{M}_{\xi}$ is defined by (\ref{1.11}). 
\end{prop}
\begin{proof}
To prove (\ref{DMRXi}), it is sufficient to consider the case that $F$ is given as $F(\Xi(\cdot)) = \prod_{m=1}^M g_m(\bd{X}(t_m))$ for $M \in \mathbb{N}$, $t_1 < \cdots < t_M \leq T \in [0, \infty)$, with symmetric bounded measurable functions $g_m$ on $\mathbb{Z}^N$, $1 \leq m \leq M$. Here we prove the equalities
\begin{equation}
\mathbb{E}_{\xi} \left[ \prod_{m=1}^M g_m(\bd{X}(t_m)) \right] = \rE_{\bd{u}} \left[ \prod_{m=1}^M g_m(\bd{V}(t_m)) \det_{1 \leq j, k \leq N} \left[ \Mxi{u_k}(T, V_j(T)) \right] \right]. \label{3.2}
\end{equation}
By Lemma \ref{WNDoob}, the LHS of (\ref{3.2}) is given by 
\begin{equation} \label{3.3}
\rE_{\bd{u}}\left[ \prod_{m=1}^M g_m(\bd{V}(t_m)) {\bf{1}}(\tau_{\bd{u}} > t_M) \frac{h(\bd{V}(t_M))}{h(\bd{u})} \right], 
\end{equation}
where we used the fact that $h(\bd{V}(t))/h(\bd{u})$ is an $\mathcal{F}_t$-martingale. 
At time $t = \tau_{\bd{u}}$, there are at least one pair $(j, j+1)$ such that $V_j(\tau_{\bd{u}}) = V_{j+1}(\tau_{\bd{u}})$, $1 \leq j \leq N-1$. We choose the minimal $j$. Let $\sigma_{j, j+1}$ be the permutation of the indices $j$ and $j+1$ and for $\bd{v}=(v_1, \cdots, v_N) \in \mathbb{Z}^N$ we put $\sigma_{j, j+1}(\bd{v}) = (v_{\sigma_{j, j+1}(k)})_{k=1}^N = (v_1, \ldots, v_{j+1}, v_j, \ldots, v_N)$. Let $\bd{u'}$ be the labeled configuration of the process at time $t = \tau_{\bd{u}}$. Since $u_j'=u_{j+1}'$ by the above setting, under the probability law $\P_{\bd{u}'}$ the processes $\bd{V}(t)$, $t > \tau_{\bd{u}}$ and $\sigma_{j, j+1}(\bd{V}(t))$, $t > \tau_{\bd{u}}$ are identical in distribution. Since $g_m$, $1 \leq m \leq M$ are symmetric, but $h$ is antisymmetric, the Markov property of the process $\bd{V}(\cdot)$ gives 
\begin{equation*}
\rE_{\bd{u}}\left[ \prod_{m=1}^M g_m(\bd{V}(t_M)){\bf{1}}(\tau_{\bd{u}} \leq t_M)\frac{h(\bd{V}(t_M))}{h(\bd{u})} \right] = 0. 
\end{equation*}
Therefore, (\ref{3.3}) is equal to 
\begin{equation*}
\rE_{\bd{u}} \left[ \prod_{m=1}^M g_m(\bd{V}(t_m))\frac{h(\bd{V}(t_M))}{h(\bd{u})} \right] = \rE_{\bd{u}} \left[ \prod_{m=1}^M g_m(\bd{V}(t_m))\frac{h(\bd{V}(T))}{h(\bd{u})} \right], 
\end{equation*}
where the $\mathcal{F}_t$-martingale property of $h(\bd{V}(t))/h(\bd{u})$ was used. By (\ref{VancS}), (\ref{3.2}) is concluded. 

Next we check that $\mathcal{M}_{\xi}$ satisfies the conditions (M1), (M2) and (M3) for Definition 1.1 in \cite{K.dif}. 
Since $\Phi_{\xi}^{u_k}(z)$ is a polynomial of $z$ of degree $N-1$, $\mathcal{M}_{\xi}^{u_k}(t, V(t))$ is expressed by a linear combination of the polynomial martingales $\{ m_n(t, V(t)) \}_{n \in \N_0}$. Then $\mathcal{M}_{\xi}(t, V(t))$, $1 \leq k \leq N$ are $\mathcal{F}_t$-martingales. Then the condition (M1) is proved. Since we assume $\xi \in \mathfrak{M}_0$, then the set of zeros of $\Phi_{\xi}^{u_k}(z)$ is different from that of $\Phi_{\xi}^{u_l}(z)$ for $k \neq l$. Therefore the condition (M2) is proved. 
By (\ref{Mlim0}) we can check that the condition (M3) is satisfied. 
Then the proof is completed. 
\end{proof}

\section{Determinantal process} 

\subsection{Correlation kernel} 
For any integer $M \in \mathbb{N}$, a sequence of times $\bd{t}=(t_1, \ldots, t_M) \in [0, \infty)^M$ with $t_1 < \cdots < t_M \leq T \in [0, \infty)$, and a sequence of continuous functions $\bd{f} = (f_{t_1}, \ldots, f_{t_M})$, the moment generating function of multitime distribution of the process $\Xi(\cdot)$ is defined by 
\begin{equation} \label{3.4}
\Psixi := \mathbb{E}_{\xi} \left[ \exp \left\{ \sum_{m=1}^M \int_{\mathbb{Z}} f_{t_m}(x) \Xi(t_m, dx) \right\} \right]. 
\end{equation}
It is expand with respect to 
\begin{equation} \label{3.5}
\chi_{t_m}(\cdot) = e^{f_{t_m}(\cdot)}-1, \quad 1 \leq m \leq M
\end{equation}
as 
\begin{equation} \label{3.6}
\Psixi = \sum_{\substack{N_m \geq 0, \\ 1 \leq m \leq M}} \sum_{\substack{\bd{x}_{N_m}^{(m)} \in \mathbb{Z}^{N_m} \cap \mathbb{W}_{N_m}, \\ 1 \leq m \leq M}} \prod_{m=1}^M \prod_{j=1}^{N_m} \chi_{t_m}\left( x_j^{(m)} \right) \rho_{\xi} \left( t_1, \bd{x}_{N_1}^{(1)}; \ldots ; t_M, \bd{x}_{N_M}^{(M)} \right), 
\end{equation}
where $\bd{x}_{N_m}^{(m)}$ denotes $(x_1^{(1)}, \ldots, x_{N_m}^{(m)})$, and (\ref{3.6}) defines the spatio-temporal correlation functions $\rho_{\xi}(\cdot)$ for the process $(\Xi(t), t \in [0, \infty), \mathbb{P}_{\xi})$. 

Given an integral kernel 
${\bf{K}}(s, x; t, y); (s, x), (t, y) \in [0, \infty) \times \mathbb{Z}$, 
the Fredholm determinant is defined as 
\begin{align}
&\Det [\delta_{st}\delta_x(y) + {\bf{K}}(s,x; t,y)\chi_t(y)] \notag \\
&= \sum_{\substack{N_m \geq 0, \\ 1 \leq m \leq M}} \sum_{\substack{\bd{x}_{N_m}^{(m)} \in \mathbb{Z}^{N_m} \cap \mathbb{W}_{N_m}, \\ 1 \leq m \leq M}} \prod_{m=1}^M \prod_{j=1}^{N_m} \chi_{t_m}\left( x_j^{(m)} \right) \det_{\substack{1 \leq j \leq N_m, 1 \leq k \leq N_n, \\ 1 \leq m, n \leq M}} \left[ {\bf{K}}(t_m, x_j^{(m)}; t_n, x_k^{(n)}) \right]. \label{3.7}
\end{align}
\begin{defi}{\rm{(Definition 1.2 in \cite{K.dif})}}
If any moment generating function (\ref{3.4}) is given by a Fredholm determinant, the process $(\Xi, \P_{\xi})$ is said to be determinantal. In this case, all spatio-temporal correlation functions are given by determinants as 
\begin{equation} \label{3.10}
\rho_{\xi}(t_1, \bd{x}_{N_1}^{(1)}; \ldots ; t_M, \bd{x}_{N_M}^{(M)}) = \det_{\substack{1 \leq j \leq N_m, 1 \leq k \leq N_n, \\ 1 \leq m, n \leq M}} \left[ \mathbb{K}_{\xi}(t_m, x_j^{(m)}; t_n, x_k^{(n)}) \right], 
\end{equation}
$0 \leq t_1 < \cdots < t_M < \infty$, $1 \leq N_m \leq N$, $\bd{x}_{N_m}^{(m)} \in S^{N_m}$, $1 \leq m \leq M \in \mathbb{N}$. Here the integral kernel, $\mathbb{K}_{\xi} : ([0, \infty) \times S)^2 \to \mathbb{R}$, is a function of initial configuration $\xi$ and is called the correlation kernel. 
\end{defi}
The main theorem of the present paper is the following. 

\begin{theorem} \label{thm.3.3}
For any initial configuration $\xi \in \mathfrak{M}_0$ with $\xi(\mathbb{Z})=N \in \mathbb{N}$, the noncolliding RW, $(\Xi(t), t \in [0, \infty), \Pxi)$ is determinantal with the kernel given by (\ref{1.14}). 
\end{theorem}
\begin{proof}
By Theorem 1.3 in \cite{K.dif}, and Proposition \ref{prop.3.1} in the present paper, we can prove that $(\Xi, \mathbb{P}_\xi)$ is determinantal with the kernel 
\begin{equation}
\label{kernelp} {\bf{K}}(s, x; t, y) = \sum_{j=1}^N p(s,x|u_j)\Mxi{u_j}(t, y) - \bd{1}(s>t)p(s-t, x|y),
\end{equation}
where $p$ is the transition probability (\ref{1.1}) and $\mathcal{M}_{\xi}$ is defined by (\ref{1.11}). 
By Lemma \ref{prob.dens} and (\ref{cSdef}) with (\ref{1.10}),  
\begin{align}
&{\bf{K}}(s, x; t, y) \notag \\
&= \sum_{j=1}^N e^{-s}I_{|x-u_j|}(s)e^t\sum_{w \in \mathbb{Z}} I_{|w-y|}(-t)\Phi_{\xi}^{u_j}(w)-\1(s>t)e^{-(s-t)}I_{|x-y|}(s-t) \notag \\
&= e^{t-s}{\bigg\{} \sum_{j=1}^N \sum_{w \in \mathbb{Z}} I_{|x-u_j|}(s)I_{|w-y|}(-t)\prod_{\substack{1 \leq \ell \leq N, \\ \ell \neq j}} \frac{w-u_{\ell}}{u_j-u_{\ell}} - \1{(s>t)}I_{|x-y|}(s-t) {\bigg\}}. \label{bfK}
\end{align}
For $w \in \{ u_j \}_{j=1}^N$, 
\begin{equation}
\prod_{\substack{1 \leq \ell \leq N, \\ \ell \neq j}} \frac{w-u_{\ell}}{u_j-u_{\ell}} = 
\begin{cases} 1, & \text{if $w = u_j$}, \\ 0, & \text{if $w=u_{\ell}$, $\ell \neq j$}. \end{cases}
\label{1or0}
\end{equation}
We apply (\ref{1or0}) to (\ref{bfK}), and then we obtain 
\begin{align}
&{\bf{K}}(s, x; t, y) \notag \\
&= e^{t-s}{\bigg\{} \sum_{j=1}^N I_{|x-u_j|}(s)I_{|y-u_j|}(-t) \notag \\
&\quad + \sum_{j=1}^N \sum_{w \in \Z \setminus \{ u_k \}_{k=1}^N} I_{|x-u_j|}(s)I_{|y-w|}(-t)\prod_{\substack{1 \leq \ell \leq N, \\ \ell \neq j}} \frac{w-u_{\ell}}{u_j-u_{\ell}} -\1(s>t)I_{|x-y|}(s-t) {\bigg\}}. \notag
\end{align}
Since any factor of the form $f(t, y)/f(s, x)$ is irrelevant for correlation kernels, we obtain (\ref{1.14}). The proof is completed. 
\end{proof}

\begin{remark}
Johansson \cite{Joh01} considered the Poissonized Plancherel measure and proved that it is a DPP. The correlation kernel is given by
\begin{equation}
\label{Joh} {\rm K}(x, y) = \sum_{k=1}^{\infty} J_{x+k}(2\sqrt{\alpha})J_{y+k}(2\sqrt{\alpha}), \quad x, y \in \Z, 
\end{equation}
with a parameter $\alpha >0$ of Poisson distribution, where $J_{\nu}(z)$ is the Bessel function related with $I_{\nu}$ by \cite{Wat44}
\begin{equation*}
I_{\nu}(z) = \begin{cases}
e^{-\nu\pi i/2} J_{\nu}(iz), & -\pi < {\rm arg}(z) < \pi/2, \\
e^{3\nu\pi i/2} J_{\nu}(iz), & \pi/2 < {\rm arg}(z) < \pi. 
\end{cases}
\end{equation*}
When $x \neq y$, (\ref{Joh}) is written as 
\begin{equation*}
{\rm K}(x, y) = \sqrt{\alpha} \frac{J_x(2\sqrt{\alpha})J_{y+1}(2\sqrt{\alpha})-J_{x+1}(2\sqrt{\alpha})J_y(2\sqrt{\alpha})}{x-y}. 
\end{equation*}
It is called the discrete Bessel kernel. See also \cite{BOO00, Sos00}. We notice that if we set $u_j = -j$, $1 \leq j \leq N$, and $s=t>0$, the first term of (\ref{1.14}) seems to provide a finite-term approximation of (\ref{Joh}) with a negative parameter $\alpha = -t^2/4$, since we see 
\begin{equation*}
\sum_{j=1}^{N} I_{x+j}(t)I_{y+j}(-t) = i^{y-x}\sum_{j=1}^{N}J_{x+j}(it)J_{y+j}(it). 
\end{equation*}
\end{remark}

\subsection{Extension to infinite particle systems} 

For $a \in \{ 2, 3, \ldots \}$, we consider a configuration on $\mathbb{Z}$ having equidistant spacing $a$ with an infinite number of particles, 
\begin{equation}
\label{1.16} \xi_{a\mathbb{Z}}(\cdot) = \sum_{k \in \mathbb{Z}} \delta_{ak}(\cdot). 
\end{equation}
For the infinite-particle configuration (\ref{1.16}), a one-parameter family of linearly independent entire functions of $z \in \mathbb{C}$ with a parameter $k \in \mathbb{Z}$ is defined by 
\begin{align}
\hatPhi(z) &:= \prod_{j \in \mathbb{Z}, j \neq k} \frac{z-aj}{ak-aj} = \prod_{n \in \mathbb{Z}, n \neq 0} \left( 1+\frac{z/a-k}{n} \right) \notag \\
&= \frac{\sin(\pi(z/a-k))}{\pi(z/a-k)} = \frac{1}{2\pi} \int_{-\pi}^{\pi} d\lambda e^{i\lambda(z/a-k)}, \quad k \in \mathbb{Z}. \label{4.1}
\end{align}
The entire functions (\ref{4.1}) are regard as the limits of polynomials (\ref{1.10}) in the sense \cite{L,KT10}, 
\begin{equation}
\hatPhi(z) = \lim_{L \to \infty} \Phi_{a\mathbb{Z} \cap [-L, L]}^{ak}(z), \quad z \in \mathbb{C}. \label{4.2}
\end{equation}
For $(t, y) \in [0, \infty) \times \mathbb{Z}$, (\ref{4.1}) defines 
\begin{align}
\hatM(t, y) &:= \cS \left[ \left. \hatPhi(W) \right| (t,y) \right] \notag \\
&= \frac{1}{2\pi} \int_{-\pi}^{\pi} d\lambda e^{-i\lambda k}\cS[ \left. e^{i\lambda W/a} \right| (t,y)] \notag \\
&= \frac{1}{2\pi}\int_{-\pi}^{\pi}d\lambda e^{i\lambda(y/a-k)}\exp \left\{ t\left( 1-\cos\frac{\lambda}{a} \right) \right\}, \quad k \in \mathbb{Z}. \label{4.3}
\end{align}
It is readily to see that if $V(t)$, $t \in [0, \infty)$ is a RW, $\hatM(t, V(t))$, $k \in \mathbb{Z}$ are $\mathcal{F}_t$-martingales. 

Let $V_j(t)$, $t \in [0, \infty)$, $j \in \mathbb{Z}$ be an infinite sequence of independent RWs. Then we have an infinite sequence of independent $\mathcal{F}_t$-martingales, 
\begin{equation} \label{4.4}
\hatM(t, V(t)), \quad k \in \mathbb{Z}, \quad t \in [0, \infty)
\end{equation}
for each $a \in \{ 2, 3, \ldots \}$ and $k \in \mathbb{Z}$. We write the labeled configuration $(aj)_{j \in \mathbb{Z}}$ with an infinite number of particles as $a\mathbb{Z}$, and under $\rP_{a\mathbb{Z}}$, $V_j(0)=aj$, $j \in \mathbb{Z}$. Then, for any $t \in [0, \infty)$, 
\begin{align}
\rE_{a\mathbb{Z}}\left[ \hatM(t, V_j(t)) \right] &= \rE_{a\mathbb{Z}}\left[ \hatM(0, V_j(0)) \right] \notag \\
&= \hatM(0, aj) \notag \\
&= \delta_{j, k}, \quad j, k \in \mathbb{Z}. \label{4.5}
\end{align}
For $n \in \N$, an index set $\{ 1, 2, \ldots, n \}$ is denoted by $\I{n}$. Given $\bd{x}=(x_1, \ldots, x_n) \in \Z^n$, when $\J = \{ j_1, \ldots, j_{n'} \} \subset \I{n}$, $1 \leq j_1 < \cdots < j_{n'} \leq n$, we put $\bd{x}_{\J} = (x_{j_1}, \ldots, x_{j_{n'}})$. 
Fix $N \in \mathbb{N}$. For $\J \subset \I{N}$, defined a determinantal martingale of (\ref{4.4})
\begin{equation} \label{4.6}
\hat{\mathcal{D}}_{a\mathbb{Z}}(t, \bd{V}_{\J}(t)) = \det_{j,k \in \J}[\hatM(t, V_j(t))], \quad t \in [0, \infty). 
\end{equation}
Let $t \in [0, \infty)$, $t \leq T \in [0, \infty)$, $N' \in \mathbb{N}$, $N'<N$, and $F_{N'}$ be a measurable function on $\mathbb{Z}^{N'}$. Then the reducibility
\begin{align}
&\sum_{\J \subset \I{N}, \#\J = N'} \rE_{a\mathbb{Z}}\left[ F_{N'}(\bd{V}_{\J}(t))\hat{\mathcal{D}}_{a\mathbb{Z}}(T, \bd{V}_N(T)) \right] \notag \\
&= \sum_{\J \subset \I{N}, \#\J = N'} \rE_{a\mathbb{Z}}\left[ F_{N'}(\bd{V}_{\J}(t))\hat{\mathcal{D}}_{a\mathbb{Z}}(T, \bd{V}_{\J}(T)) \right] \notag \\
&= \int_{\mathbb{W}_{N'}} \xi^{\otimes N'}_{a\mathbb{Z}}(d\bd{v})\rE_{\bd{v}}\left[ F_{N'}(\bd{V}_{N'}(t))\hat{\mathcal{D}}_{a\mathbb{Z}}(T, \bd{V}_{N'}(T)) \right]. \label{4.7}
\end{align}
holds. The proof is the same as that for Lemma 2.1 in \cite{K.dif}, where the martingale property (\ref{4.5}) plays an essential role. Note that the last expression of (\ref{4.7}) does not change even if we replace $N$ in the LHS by any other integer $\tilde{N}$ with $\tilde{N} > N$. Based on such consistency in reduction of DMRs and the fact (\ref{4.2}), the noncolliding RW with an infinite number of particles started at $\xi_{a\mathbb{Z}}$ is defined as follows \cite{K.dif}. 

\begin{defi} \label{def.4.1}
For each $a \in \{ 2, 3, \ldots \}$, the noncolliding RW started at $\xi_{a\mathbb{Z}}$, denoted as $(\Xi(t), t \in [0, \infty), \P_{\xi_{a\mathbb{Z}}})$, is defined by the following. Let $t \in [0, \infty)$, $t \leq T \in [0, \infty)$. For any $\mathcal{F}(t)$-measurable bounded function $F$, which depends at most $n$ paths of RWs, $n \in \mathbb{N}$, and is symmetric at each time $s \leq t$, $s \in [0, \infty)$, its expectation is given by 
\begin{equation} \label{4.8}
\E_{\xi_{a\mathbb{Z}}}[F(\Xi(\cdot))] = {\rm{E}}_{a\mathbb{Z}} \left[ F\left( \sum_{j=1}^n \delta_{V_j(\cdot)} \right) \hat{\mathcal{D}}_{a\mathbb{Z}}(T,  \bd{V}_{\I{n}}(T)) \right]. 
\end{equation}
\end{defi}
Equation (\ref{4.8}) says that the noncolliding RW has DMR, hence we can characterize this infinite particle system $(\Xi(t), t \in [0, \infty), \P_{\xi_{a\mathbb{Z}}})$ as follows. 
\begin{prop} \label{prop.4.2}
The noncolliding RW, $(\Xi(t), t \in [0, \infty), \P_{\xi_{a\mathbb{Z}}})$, $a \in \{ 2, 3, \ldots \}$ is determinantal with the correlation kernel given by (\ref{4.9}). 
\end{prop}
\begin{proof}
We omit the irrelevant factor $e^{t-s}$ in $\sum_{j \in \mathbb{Z}} p(s,x|aj)\hatMj (t, y) - \bd{1}(s>t)p(s-t, x|y)$ and obtain (\ref{4.9}). 
\end{proof}

\subsection{Relaxation phenomenon} 

In order to state the theorem, we define a DPP. 

\begin{defi} \label{def.4.3}
For a given density $0< \rho < 1$, the probability measure $\mu_{\rho}^{\sin}$ on $\Z$ is defined as a DPP with the sine kernel 
\begin{equation*}
\bK_{\rho}^{\sin}(y-x) = \frac{\sin(\rho\pi(y-x))}{\pi(y-x)}.
\end{equation*}
\end{defi}

\begin{theorem} \label{thm.4.4}
For each $a \in \{ 2, 3, \ldots \}$, the process $(\Xi(t), t \in [0, \infty), \P_{\xi_{a\Z}})$ starting from the configuration (\ref{xiaZ}) shows a relaxation phenomenon to the stationary process $(\Xi(t), t \in [0, \infty), \P_{\rho})$ with $\rho=1/a$. The stationary process $(\Xi(t), t \in [0, \infty), \P_{\rho})$ is reversible with respect to $\mu_{\rho}^{\sin}$ and is determinantal with the correlation kernel given by 
\begin{equation}
\label{1.17} {\bf{K}}_{\rho}(t-s, y-x) = \begin{cases} \vspace{12pt}
\displaystyle \int_0^{\rho} du \cos(u\pi(y-x)) e^{-(t-s)\cos u\pi}, &{\rm{if}} \: s<t, \\

\vspace{12pt}

\displaystyle \frac{\sin(\rho\pi(y-x))}{\pi(y-x)}, &{\rm{if}} \: s=t, \\

\displaystyle -\int_{\rho}^{1} du \cos(u\pi(y-x)) e^{-(t-s)\cos u\pi}, &{\rm{if}} \: s>t. \end{cases}
\end{equation}
\end{theorem}
\begin{proof}
We rewrite (\ref{1.1}) as follows, 
\begin{align}
p(t, y|x) &= \frac{1}{2\pi}\int_{-\pi}^{\pi}dk e^{ik(y-x)} e^{-(1-\cos k)t} \notag \\
&= \frac{1}{2\pi a} \int_{-a\pi}^{a\pi} d\theta e^{i\theta(y-x)/a} \exp\left\{ -\left( 1-\cos \left( \frac{\theta}{a} \right) \right) t \right\} \notag \\
&= \int_0^1 du \cos(u\pi(y-x)) \exp\left\{ -\left( 1-\cos(u\pi)  \right) t \right\} \label{4.10}
\end{align}
for $t \in [0, \infty)$, where $a \in \mathbb{N}$. Then we have 
\begin{align*}
&\sum_{j \in \mathbb{Z}}p(s, x|aj) \hatMj(t, y) \\
&= \frac{1}{4\pi^2a} \sum_{j \in \mathbb{Z}} \int_{-a\pi}^{a\pi}d\theta \int_{-\pi}^{\pi}d\lambda e^{i(\theta x+\lambda y)/a} e^{-i(\theta+\lambda)j} \exp\left\{ -s\left( 1-\cos ( \theta/a ) \right) +t( 1-\cos(\lambda /a) \right\}.
\end{align*}

For (\ref{4.9}), $(s, x), (t, y) \in [0, \infty) \times \mathbb{Z}$, we put
\begin{equation*} \label{4.11}
\mbK_{\xi_{a\Z}}(s,x;t,y)+{\bf{1}}(s>t)p(s-t,x|y)=G(s,x;t,y)+R(s,x;t,y)
\end{equation*}
with 
\begin{align*} 
G(s,x;t,y) &= \frac{1}{4\pi^2a} \int_{|\theta| \leq \pi} d\theta \\
&\quad \times \int_{|\lambda| \leq \pi} d\lambda e^{i(\theta x+\lambda y)/a + (t-s)(1-\cos(\lambda /a)) } \sum_{j \in \mathbb{Z}} e^{-i(\theta+\lambda)j} e^{s\{ \cos(\theta /a)-\cos(\lambda /a) \}}, 
\end{align*}
and 
\begin{align*} 
R(s,x;t,y) &= \frac{1}{4\pi^2a} \sum_{j \in \Z} \int_{\pi < |\theta| \leq  a \pi} d\theta \\
&\quad \times \int_{|\lambda| \leq \pi} d\lambda e^{i(\theta x+\lambda y)/a + (t-s)(1-\cos(\lambda /a)) } e^{-i(\theta+\lambda)j} e^{s\{ \cos(\theta /a)-\cos(\lambda /a) \}}.
\end{align*}
Since $\sum_{j \in \Z}e^{-i(\theta+\lambda)j}= 2\pi\delta(\theta+\lambda)$ for $\theta, \lambda \in (-\pi, \pi]$, we obtain 
\begin{equation*} \label{4,14}
G(s,x; t,y) = \frac{1}{2\pi a} \int_{-\pi}^{\pi} d\lambda e^{i\lambda(y-x)/a+(t-s)(1-\cos(\lambda /a))} =: \mathcal{G}(t-s, y-x).
\end{equation*}
On the other hand, when $\pi < |\theta| \leq a\pi$ and $|\lambda| \leq \pi$, $\cos(\theta /a) < \cos(\lambda /a)$. We get 
\begin{equation*} 
e^{ \cos(\theta /a)-\cos(\lambda /a) } < 1
\end{equation*}
Then for any fixed $s, t \in (0, \infty)$, 
\begin{equation*}
|R(s+\tau, x; t+\tau,y)| \to 0 \quad \text{as} \: \tau \to \infty
\end{equation*}
uniformly on any $(x, y) \in \Z^2$ and it implies
\begin{equation}
\label{kconv} \mbK_{\xi_{a\Z}}(s+\tau, x; t+\tau, y) \to \bK_{\rho}(t-s, y-x) \quad \text{as} \: \tau \to \infty, 
\end{equation}
where
\begin{align*}
\bK_{\rho}(t-s, y-x) &= \mathcal{G}(t-s, y-x) - {\bf{1}}(s>t)p(s-t, x|y) \notag \\
&= \frac{1}{2\pi} \int_{-\rho \pi}^{\rho \pi} d\lambda e^{i\lambda(y-x)+(t-s)(1-\cos\lambda)}-{\bf{1}}(s>t)p(s-t, x|y). 
\end{align*}
This is equal to (\ref{1.17}) up to an irrelevant factor $e^{t-s}$. 
Here we remark $\rho = 1/a$ gives the particle density on $\mathbb{Z}$. 
The convergence of the correlation kernel (\ref{kconv}) implies the convergence of generating function for correlation functions $\Psi_{\xi_{a\Z}}^{\bd{t}}[\bd{f}]$, and thus the convergence of the determinantal process to an equilibrium determinantal process. Thus the proof is completed. 
\end{proof}

This is an example of relaxation phenomena discussed in \cite{K.dif,K.non,KT09,KT10,KT11}. 

\begin{remark}
If the initial configuration is $\xi_{\Z}$, all sites are occupied and there occurs no time-evolution in the present system. In this case we have $\mathbb{K}_{\xi_{\Z}}(s,x;t,y)=I_{|x-y|}(s-t)\1(s \leq t)$. Since $I_n(0)=\delta_{n, 0}$, $n \in \N$, it gives a trivial result, $\rho_{\xi_{\Z}} \equiv 1$. 
\end{remark}


\vspace{12pt}

\noindent
{\bf{Acknowledgements}} \quad The present author would like to express his thanks to Makoto Katori for his valuable suggestion and constant encouragement. He also would like to thank Hideki Tanemura for his many valuable comments. 




\begin{thebibliography}{99}
\bibitem{BOO00} 
Borodin, A., Okounkov, A., Olshanski, G.: 
Asymptotics of Plancherel measures for symmetric groups. J. Amer. Math. Soc. {\bd 13}, 481-515 (2000)

\bibitem{DFL13}
Decreusefond, L., Flint, I., Low, K. C.:
Perfect simulation of determinantal point processes.
{\sf arXiv:math.PR/1311.1027}

\bibitem{Dys62}
Dyson, F. J. :
A Brownian-motion model for the eigenvalues of a random matrix.
J. Math. Phys. {\bf 3}, 1191-1198 (1962)

\bibitem{EK08} 
Eichelsbacher, P., K\"onig, W.:
Ordered random walks.
Electron. J. Probab. {\bf 13}, no.46, 1307-1336 (2008)

\bibitem{For10}
Forrester, P. J.:
Log-gases and Random Matrices.
London Mathematical Society Monographs, Princeton:
Princeton University Press, 2010

\bibitem{Gin65}
Ginibre, J.:
Statistical ensembles of complex, quaternion, and real matrices.
J. Math. Phys. {\bf 6}, 440-449 (1965)

\bibitem{Gol10}
Goldman, A.:
The Palm measure and the Voronoi tessellation 
for the Ginibre process.
Ann. Appl. Probab. 
{\bf 20}, 90-128 (2010)

\bibitem{Joh01} 
Johansson, K.: 
Discrete orthogonal polynomial ensembles and the Plancherel measure. Ann. Math. {\bf 153}, 259-296 (2001)

\bibitem{K.dif} 
Katori, M. : 
Determinantal martingales and noncolliding diffusion processes.
Stochastic Process. Appl. {\bf 124}, 3724-3768 (2014)

\bibitem{K.non} 
Katori, M.: 
Determinantal martingales and correlations of noncolliding random walks; \textsf{arXiv:math.PR/1307.1856}

\bibitem{KT09} 
Katori, M., Tanemura, H.: 
Zeros of Airy function and relaxation process. J. Stat. Phys. {\bf{136}}, 1177-1204 (2009)

\bibitem{KT10} 
Katori, M., Tanemura, H.: 
Non-equilibrium dynamics of Dyson's model with an infinite number of particles. Commun. Math. Phys. {\bf{293}}, 469-497 (2010)

\bibitem{KT11} 
Katori, M., Tanemura, H.: 
Noncolliding squared Bessel processes. J. Stat. Phys. {\bf{142}}, 592-615 (2011)

\bibitem{KT13} 
Katori, M., Tanemura, H.: 
Complex Brownian motion representation of the Dyson model. Electron. Commun. Probab. {\bf{18}}, no.4, 1-16 (2013)

\bibitem{Koe05}
K\"onig, W.:
Orthogonal polynomial ensembles in probability theory.
Probab. Surveys {\bf 2}, 385-447 (2005)

\bibitem{KOR} 
K\"{o}nig, W., O'Connell, N., Roch, S.: 
Non-colliding random walks, tandem queues, and discrete orthogonal polynomial ensembles. Electron. J. Probab. {\bf{7}}, 1-24 (2002)

\bibitem{KT12}
Kulesza, A., Taskar, B.:
Determinantal point processes for machine learning.
Foundations and Trends in Machine Learning, vol. 5, issue 2-3, pp.123-286, 2012

\bibitem{L} 
Levin, B. Ya.: 
Lectures on Entire Functions. Translations of Mathematical Monographs, {\bf{150}}, Province R. I. : Amer. Math. Soc., 1996

\bibitem{Meh04}
Mehta, M. L.:
Random Matrices.  third ed., Amsterdam: Elsevier, 2004

\bibitem{MS12}
Miyoshi, N., Shirai, T.:
A cellular network model with Ginibre configured base stations.
Adv. Appl. Probab. {\bf 46}, 832-845 (2014)


\bibitem{Osa12}
Osada, H.: 
Infinite-dimensional stochastic differential equations related to random matrices. Probab. Theory Related Fields {\bf 153}, 471-509 (2012)

\bibitem{Osa13a}
Osada, H.: 
Interacting Brownian motions in infinite dimensions with logarithmic interaction potentials. Ann. Probab. {\bf 41}, 1-49 (2013)

\bibitem{Osa13b}
Osada, H.: 
Interacting Brownian motions in infinite dimensions with logarithmic interaction potentials II: Airy random point field. Stochastic Process. Appl. {\bf 123}, 813-838 (2013)

\bibitem{Sat99} 
Sato, K.: 
L\'{e}vy processes and Infinitely Divisible Distributions. Cambridge, U.K.: Cambridge University Press, 1999

\bibitem{ST03}
Shirai, T., Takahashi, Y.: 
Random point fields associated with certain
Fredholm determinants I:
fermion, Poisson and boson point process.
J. Funct. Anal.
{\bf 205}, 414-463 (2003)

\bibitem{Sos00}
Soshnikov, A. : 
Determinantal random point fields.
Russian Math. Surveys {\bf 55}, 923-975 (2000)

\bibitem{Wat44} 
Watson, G.N.: 
A Treatise on the Theory of Bessel Functions, 2nd ed., Cambridge, U.K: Cambridge Univ. Press., 1944
\end{thebibliography}
\end{document}